\documentclass[11pt,b5paper,twoside, headrule]{amsart}

\usepackage{amsfonts, amsmath, amssymb,latexsym}
\usepackage{epsfig}
\usepackage[curve]{xy}
\usepackage{algorithmic}
\usepackage{algorithm}
\usepackage{enumerate}
\usepackage{framed}
\usepackage{hyperref}
\usepackage{mathtools}

\usepackage[OT2,T1]{fontenc}

\usepackage{scalerel,amssymb}

\footskip=10mm\parskip 0.2cm\addtocounter{page}{0}
\newtheorem{thm}{Theorem}[section]
\newtheorem{cor}[thm]{Corollary}
\newtheorem{lem}[thm]{Lemma}
\newtheorem{prop}[thm]{Proposition}

\theoremstyle{definition}
\newtheorem{defn}[thm]{Definition}

\newtheorem{notation}[thm]{Notation}

\theoremstyle{remark}

\numberwithin{equation}{section}

\DeclareSymbolFont{cyrletters}{OT2}{wncyr}{m}{n}
\DeclareMathSymbol{\Sha}{\mathalpha}{cyrletters}{"58}

\newcommand{\Z}{{\mathbb Z}}

\newcommand{\Q}{{\mathbb Q}}

\title[Potentially Diagonalizable modular lifts of large weight]{Potentially Diagonalizable modular lifts of large weight}

\author{\sc Iv\'an Blanco-Chac\'on}
\address{Department of Physics and Mathematics\\
School of Science\\
University of Alcal\'a de Henares\\
Madrid\\
Spain}
\email{ivan.blancoc@uah.es}
\thanks{Partially supported by MTM2016-79400-P}

\author{\sc Luis Dieulefait}
\address{Department of Algebra\\
Faculty of Mathematics and Computer Science\\
University of Barcelona\\
Barcelona\\
Spain}
\email{ldieulefait@ub.edu}

\begin{document}
\renewcommand\baselinestretch{1.2}
\renewcommand{\arraystretch}{1}
\def\base{\baselineskip}
\font\tenhtxt=eufm10 scaled \magstep0 \font\tenBbb=msbm10 scaled
\magstep0 \font\tenrm=cmr10 scaled \magstep0 \font\tenbf=cmb10
scaled \magstep0


\def\evenhead{{\protect\centerline{\textsl{\large{I. Blanco}}}\hfill}}

\def\oddhead{{\protect\centerline{\textsl{\large{On the non vanishing of the cyclotomic $p$-adic $L$-functions}}}\hfill}}

\pagestyle{myheadings} \markboth{\evenhead}{\oddhead}

\thispagestyle{empty}

\maketitle

\begin{abstract}We prove that for a Hecke cuspform $f\in S_k(\Gamma_0(N),\chi)$ and a prime $l>\max\{k,6\}$ such that $l\nmid N$, there exists an infinite family $\{k_r\}_{r\geq 1}\subseteq\mathbb{Z}$ such that for each $k_r$, there is a cusp form  $f_{k_r}\in S_{k_r}(\Gamma_0(N),\chi)$ such that the Deligne representation $\rho_{f_{k_r,l}}$ is a crystalline and potentially diagonalizable lift of $\overline{\rho}_{f,l}$. When $f$ is $l$-ordinary, we base our proof on the theory of Hida families, while in the non-ordinary case, we adapt a local-to-global argument due to Khare and Wintenberger in the setting of their proof of Serre's modularity conjecture, together with a result on existence of lifts with prescribed local conditions over CM fields, a flatness result due to B\"ockle and  a local dimension result by Kisin. We discuss the motivation and tentative future applications of our result in ongoing research on the automorphy of $\mathrm{GL}_{2n}$-type representations in the higher level case.
\end{abstract}

In \cite{gee}, the authors introduce potentially diagonalizable representations, a tool which allows to prove automorphy in a wide range of cases. Since then, the generality and versatility of the notion has brought new hopes in the study of Langlands functoriality. More precisely, we are referring to the part of Langlands functoriality conjecturing the preservation of automorphy under standard group-theoretical operations such as restriction, extension, symmetric powers or tensor products. For instance, in \cite{luissara}, the second author has studied $\mathrm{GL}_{2n}$-type representations of the form $\pi_f\otimes\pi_{\phi}$ with $f$ modular of level $1$ and $\phi$ automorphic and self-dual, and the study of similar results for higher level in the $\mathrm{GL}_2$ factor, at least in the odd case, is an ongoing work. 

A promising strategy towards this end is to produce \emph{safe} chains of congruences along a family of modular representations of varying Serre weights and level until we fall in a \emph{tractable} case, for which we know that automorphy is preserved after tensoring. Such a tractable case occurs, for instance, when the final modular form is CM. This approach, together with the fact that potential diagonalizability is preserved under tensor product is a motivation for the present work. In fact, the main result in this paper, which produces congruences of modular forms of arbitrarily large weight which are known to be potentially diagonalizable locally at $p$, is a very helpful tool for attacking Langlands functoriality via the method of \emph{safe} chains of congruences: the novelty in our result with respect to classical ones is that we can preserve the potentially diagonalizable character at $p$, which is key to allow the automorphy lifting theorems in \cite{gee} to be applied through the process of proving functoriality. And by the way, this is what we mean when we call the chain of congruences \emph{safe}: that at each congruence in the chain appropriate conditions hold so that it is amenable for some automorphy lifting theorem to be applied.

Let $f\in S_k(\Gamma_0(N),\chi)$ be a Hecke-eigenform for each Hecke operator $T_p$ with $p\nmid N$, with field of coefficients $K_f$, and let $l>k$ be a prime with $l\nmid N$. Denote by $\mathcal{O}_f$ the ring of integers of $K_f$. Further on, in our argument we will need to assume that $l>6$, leaving the remaining cases for a further study. The reason is that we use in a crucial manner a general automorphy lift result in \cite{gee} which we recall as Theorem \ref{geetool}.

Fix a prime $\lambda$ of $\mathcal{O}_f$ above $l$ and denote by $\mathcal{O}_{f,\lambda}$ the $\lambda$-adic completion of $\mathcal{O}_f$. Let $\rho_{f,l}:G_{\mathbb{Q}}\to\mathrm{GL}_2(\mathcal{O}_{f,\lambda})$ the $l$-adic representation attached to $f$ by Deligne and let $\overline{\rho}_{f,l}:G_{\mathbb{Q}}\to\mathrm{GL}_2(\mathbb{F})$ be the reduction of $\rho_{f,l}$ modulo $\lambda$, where $\mathbb{F}$ is the residual field of $\lambda$. Recall that $\det(\rho_{f,l})=\chi\epsilon_l^{k-1}$, where $\epsilon_l:G_{\mathbb{Q}}\to\mathcal{O}_{f,\lambda}^*$ is the $l$-adic cyclotomic character, which again, we recall, $\epsilon_l(\phi_p)=p$ for $p\nmid Nl$, where $\phi_p$ stands for any Frobenius element at $p$.

For technical reasons, we need to assume that $\mathrm{SL}_2(\mathbb{F}_l)\subseteq\mathrm{Im}(\overline{\rho}_{f,l})$. It is known after Momose (\cite{momose}) that this is the case for all but finitely many $l$ if $f$ is not CM, which is a natural assumption to make, and which we will also make. In this case, it also follows that $\overline{\rho}_{f,l}$ is absolutely irreducible, hence in particular, simple, which avoids us working with semisimplifications when dealing with the residual representation at several points in our study. Let $k_0$ be the Serre weight of $\overline{\rho}_{f,l}$. Since $l\nmid N$ and $2\leq k\leq l-1$, it follows that $k=k_0$.

The aim of the present work is to prove the following

\begin{thm}Let $f\in S_k(\Gamma_0(N),\chi)$, $K_f$, $l$ and $\lambda$ be as in the above paragraph. Then, there exists an infinite sequence $\{k_r\}_{r\geq 1}$, and for each $k_r$, a cusp form $f_{k_r}\in S_{k_r}(\Gamma_0(N),\chi)$  such that $\rho_{f_{k_r,l}}$ is a lift of $\overline{\rho}_{f,l}$ and such that $\rho_{f_{k_r,l}}|_{G_{\mathbb{Q}_l}}$ is 
\begin{itemize}
\item[1.] crystalline of Hodge-Tate weight $\{0,k_r-1\}$,
\item[2.] potentially diagonalizable, in the sense of \cite{gee} (we recall the precise definition in the first section).
\end{itemize}
\end{thm}

Both in the ordinary and in the non-ordinary case, the representation $\rho_{f,l}|_{G_{\mathbb{Q}_l}}$ is potentially diagonalizable, due to \cite{gee} Lemma 1.4.3, and this is the starting point for our proof. In the ordinary case, we use Hida families  and examine the representations attached to the specializations. These are crystalline and ordinary, hence potentially diagonalizable.

The proof in the non-ordinary case is significantly more involved and uses, first, an infinite family of local lifts (via \cite{blgg} Lemmas 4.1.15 and 4.1.19) of the restriction of $\overline{\rho}_{f,l}$ to $I_l$, the inertia group at $l$. Then, we mimick an argument due to Khare and Wintenberger (\cite{kwannals}): from the universal deformation rings attached to suitable deformation conditions, we construct a global deformation ring $\mathcal{R}_{k_r,global}$ and prove that its dimension is greater than or equal to $1$ and that it is finitely generated as a module over a suitable ring of integers. Hence, by an argument due to B\"ockle in \cite{bockle}, $\mathrm{Spec}(\mathcal{R}_{k_r,global})$ contains $\overline{\mathbb{Q}}_l$-points which, moreover, we prove to be modular.

\section{Definitions, notations and preliminary facts}

Throughout this paper, $K$ will denote a finite extension of $\mathbb{Q}_l$ and $\mathcal{O}$ its ring of integers. Fix algebraic closures $\overline{\mathbb{Q}}$ of $\mathbb{Q}$, $\overline{K}$ of $K$ and $\overline{\mathbb{Q}}_l$ of $\mathbb{Q}_l$ respectively, fix an embedding of $\overline{\mathbb{Q}}$ in $\overline{\mathbb{Q}}_l$ and denote as usual $G_K=\mathrm{Gal}(\overline{K}/K)$, the absolute Galois group of $K$. Given a local ring $R$ of residual characteristic $l$, denote by $\mathfrak{m}_R$ its unique maximal ideal and by $\mathbb{F}$ its residual field. Denote by $\mathrm{GL}_m(R)_1$ the kernel of the reduction map $\mathrm{GL}_m(R)\to \mathrm{GL}_m(\mathbb{F})$. For a representarion $\rho:\Gamma\to \mathrm{GL}_m(R)$ and for $g\in \mathrm{GL}_m(R)$, we will denote $\mathrm(int)(g)\rho:=g\rho g^{-1}$, the conjugation of $\rho$ by $g$. If $R=\mathcal{O}_{\overline{\mathbb{Q}}_l}$, denote by $\overline{\rho}$ its residual representation obtained by reducing modulo $\mathfrak{m}_{\overline{\mathbb{Q}}_l}$. Finally, denote by $\det(\rho)$ and $\det(\overline{\rho})$ the compositions of $\rho$ and $\overline{\rho}$ with the determinant map.

Let $\Gamma=G_K$. Denote by $\rho^{\square}_{\mathcal{O}}:\Gamma\to \mathrm{GL}_n(R^{\square}_{\mathcal{O},\overline{\rho}})$ the framed universal deformation of $\overline{\rho}$ to the framed universal deformation ring $R^{\square}_{\mathcal{O},\overline{\rho}}$. Framed deformations are just lifts of the residual representation, with no other equivalence than equality (a basis is fixed for the residual representation, compatibly lifted in each deformation). For each finite set $S$ of rational places (possibly containing the archimedean one), denote by $R_{\mathcal{O},S}^{\square}$ the complete noetherian local $\mathcal{O}$-algebra which represents the functor which sends $A$ to the set of isomorphism classes of pairs $(\rho_{f,A}, \{\beta_{A,\nu}\}_{\nu\in S})$, with $\rho_{f,A}$ framed deformation of $\overline{\rho}_{f,l}$ to $A$ (complete noetherian local $\mathcal{O}$-algebra with residue field $\mathbb{F}$) unramified outside $S$, and for each $\nu\in S$, $\beta_{A,\nu}$ is a lift of the chosen basis of $\overline{\rho}_{f,l}$. Setting $\mu=\chi\epsilon_l^{k-1}=\det(\rho_{f,l})$, we have universal objects $R^{\square, \mu}_{\mathcal{O},\overline{\rho}}$ and $R_{\mathcal{O},S}^{\square, \mu}$ parametrizing framed deformations of $\overline{\rho}$ with fixed determinant $\mu$. To make \emph{having fixed determinant $\mu$} a proper deformation condition in the sense of Mazur, we slightly abuse notation when saying that $\overline{\rho}_{f,l}$ satisfies the condition by identifying $\epsilon_l$ with its composition with reduction modulo $l$.

Denote by $\rho_{\mathcal{O}}:\Gamma\to \mathrm{GL}_n(R_{\mathcal{O},\overline{\rho}})$ the  universal deformation of $\overline{\rho}$ to the universal deformation ring $R_{\mathcal{O},\overline{\rho}}$ (whenever the deformation functor is representable) and the fixed-determinant (uni)versal object $R_{\mathcal{O}, \overline{\rho}}^{\mu}$ for $\mu=\psi\epsilon_l^n$, with $\psi$ of finite order and $n\in\mathbb{Z}$. We recall the following definition:

\begin{defn}[\cite{gee}, p. 25] The ring $R^{\square}_{\mathcal{O},\overline{\rho},\{H_{\tau}\}, F-cris}$ is defined as the unique reduced and $l$-torsion free quotient of $R^{\square}_{\mathcal{O},\overline{\rho}}$  such that a point $\xi: R^{\square}_{\mathcal{O},\overline{\rho}}\to\overline{\mathbb{Q}}_l$ factors through $R^{\square}_{\mathcal{O},\overline{\rho},\{H_{\tau}\},\mathcal{P}}$ if and only if it corresponds to a representation $\rho:G_K\to \mathrm{GL}_n(\overline{\mathbb{Q}}_l)$ which is de Rham with Hodge-Tate numbers $\{H_{\tau}\}$ for all $\tau:K\hookrightarrow\overline{\mathbb{Q}_l}$ and which is crystalline after restriction to $F$. Observe that we are imposing no condition on the determinant here.
\end{defn}

The existence of $R^{\square}_{\mathcal{O},\overline{\rho},\{H_{\tau}\},F-cris}$ follows from Theorem 2.7.6. of \cite{kisin}: indeed, given $R^{\square}_{\mathcal{O},\overline{\rho}}$,  Corollary 2.6.2 of loc. cit. yields a quotient $R^{\square}_{\mathcal{O},\overline{\rho}, \{H_{\tau}\}}$, corresponding to potentially semistable representations of Hodge-Tate type $\{H_{\tau}\}$. Then,  Corollary 2.7.7 of loc. cit., yields $R^{\square}_{\mathcal{O},\overline{\rho},\{H_{\tau}\},F-cris}$ as the quotient of $R^{\square}_{\mathcal{O},\overline{\rho}, \{H_{\tau}\}}$ defined by the equation $N=0$, where $N$ is as in Theorem 2.5.5 (2) of loc. cit.

\begin{defn}[Barnet-Lamb, Gee, Gerarghty, Taylor, \cite{gee}, pag. 26] For any two representations $\rho_1,\rho_2: G_K\to \mathrm{GL}_n(\mathcal{O}_{\overline{\mathbb{Q}_l}})$, we say that $\rho_1$ connects to $\rho_2$ if
\begin{itemize}
\item $\overline{\rho_1}$ and $\overline{\rho_2}$ are equivalent,
\item $\rho_1$ and $\rho_2$ are potentially crystalline,
\item for each continuous field embedding $\tau: K\hookrightarrow\overline{\mathbb{Q}_l}$,  $HT_{\tau}(\rho_1)=HT_{\tau}(\rho_2)$,
\item and $\rho_1$ and $\rho_2$ define points on the same irreducible component of the scheme $Spec(\mathcal{R}^{\square}_{\overline{\rho}_1,\{HT_{\tau}(\rho_1)\},F-cris}\otimes\overline{\mathbb{Q}_l})$ for some (and hence all) sufficiently large $F$.
\end{itemize}
\end{defn}
It is convenient to observe that \emph{connects} is an equivalence relation and the fact that $\rho_1$ connects with $\rho_2$ does not depend only on the $\mathrm{GL}_n(\mathcal{O}_{\overline{\mathbb{Q}}_l})$-conjugacy class of $\rho_1$ or $\rho_2$; it depends on $\rho_1$ and $\rho_2$ themselves, and on their common connected component (cf. \cite{gee} Lemma 1.2.2).

\begin{defn}A Galois representation $\rho:G_K\to \mathrm{GL}_n(\mathcal{O}_{\overline{\mathbb{Q}_l}})$ is said to be diagonalizable if it is crystalline and connects to some representation $\chi_1\oplus...\oplus\chi_n$, with $\chi_i:G_K\to \mathcal{O}_{\overline{\mathbb{Q}_l}}^*$ crystalline characters, and it is said to be potentially diagonalizable if there exists a finite extension $K'/K$ such that $\rho|_{G_{K'}}$ is diagonalizable.
\end{defn}

In several arguments in Section 3, we will crucially use the following potential diagonalizability criteria :

\begin{thm}[\cite{gee}, Lemma 1.4.3] Let $\rho:G_K\to \mathrm{GL}_n(\overline{\mathbb{Q}}_l)$ be a potentially crystalline representation (i.e. crystalline after restriction to $G_F$ for some finite extension $F/K$).
\begin{itemize}
\item[a)] If $\rho$ has a $G_K$-invariant filtration with one-dimensional graded pieces, then $\rho$ is potentially diagonalizable. In particular, potentially crystalline ordinary representations are potentially diagonalizable.
\item[b)] If $K/\mathbb{Q}_l$ is unramified, $\rho$ is crystalline and all the Hodge-Tate weights are in the range $\{0,...,l-2\}$ (i.e. the Fontaine-Lafaille range), then $\rho$ is potentially diagonalizable.
\end{itemize}
\label{thmbase}
\end{thm}

In the rest of the article, it will be $n=2$ unless we state the contrary.

\section{The ordinary case} Assume $f$ is $l$-ordinary so that, as is well known, $\rho_{f,l}$ is ordinary in the sense of Theorem \ref{thmbase} a). A simple strategy to prove our main theorem in this case is by using Hida families. Next, we introduce the precise definitions and relevant facts, enough for our purposes. First of all, we need to recall the concept of ordinary $l$-stabilisation:

\begin{defn}Let $f\in S_k(\Gamma_1(N),\chi)$ be an ordinary $T_l$-eigenform and $a_l(f)$ its eigenvalue. Consider the Hecke characteristic polynomial $X^2-a_l(f)X+\chi(l)l^{k-1}$ and label its roots $\alpha$ and $\beta$ so that $\mathrm{ord}_l(\alpha)=0$ and $\mathrm{ord}_l(\beta)=k-1$. The $l$-ordinary stabilisation of $f$ is the modular form whose $q$-expansion is given by
$$
f^{(l)}(q)=f(q)-\beta f(q^l).
$$
This is a modular form in the space $S_k(\Gamma_1(N)\cap\Gamma_0(l))$.
\end{defn}

Set $\Gamma = 1 + lN\mathbb{Z}_l$ and let $\Lambda = \mathcal{O}[[\Gamma]]$ be the completed group ring of $\Gamma$. The weight space is defined to be
$$
\Omega:= \mathrm{Hom}_{\mathcal{O}-alg}(\Lambda, \mathcal{O})\cong \mathrm{Hom}_{cts}(\Gamma, \mathcal{O}^*).
$$
The subset of classical characters of $\Omega$ is
$$
\Omega^{cl}=\{\chi_k := (\gamma\mapsto\gamma^k),\mbox{ with  }k\in\mathbb{Z}_{\geq 2} \}.
$$
Given any finite flat extension $\Lambda'$ of $\Lambda$, denote $\Omega'= \mathrm{Hom}(\Lambda',\mathcal{O})$. This space is endowed with a natural projection $\kappa:\Omega'\to\Omega$, induced by the inclusion $\Lambda\hookrightarrow\Lambda'$.

\begin{defn}[Darmon-Rotger, \cite{dr} p. 22] A Hida family of tame level $N$ is a quadruple $\bold{f}=(\Lambda_{\bold{f}},\Omega_{\bold{f}},\Omega_{\bold{f}}^{cl},\bold{f}(q))$ where
\begin{itemize}
\item[a)] $\Lambda_{\bold{f}}$ is a finite and flat extension of $\Lambda$.
\item[b)] $\Omega_{\bold{f}}$ is a non-empty open subset of $X_{\bold{f}} := \mathrm{Hom}(\Lambda_{\bold{f}} , \mathbb{C}_p)$ and $\Omega_{\bold{f}}^{cl}$ is an $l$-adically dense
subset of $\Omega_{\bold{f}}$ whose image under $\kappa$ lies in $\Omega^{cl}$,
\item[c)] $\displaystyle \bold{f}(q)= \sum_{n\geq 1}\bold{a}_nq^n\in\Lambda_{\bold{f}} [[q]]$ is a formal $q$-series such that, for all $x\in\Omega_{\bold{f}}^{cl}$, the power series
$$
f^{(l)}_x:= \sum_{n=1}^{\infty}a_n(x)q^n
$$
is the $q$-expansion of the ordinary $l$-stabilization of a normalised newform (denoted $f_x$) of weight $\kappa(x)$ on $\Gamma_1(N)$.
\end{itemize}
\label{defnHida}
\end{defn}

With this given definition, one has the following result:

\begin{thm}[\cite{dr} p. 23] Let $f$ be an ordinary newform in $S_k(N, K_f )$.  There exists a Hida family $(\Lambda_{\bold{f}},\Omega_{\bold{f}},\Omega_{\bold{f}}^{cl},\bold{f}(q))$  of tame level $N$ and a classical point $x_k\in\Omega_{\bold{f}}^{cl}$ such that $\kappa(x_k)=k$  and $f_{x_k}= f$.
\label{hidapaso}
\end{thm}

Moreover, one of the features which Hida families carry is that of the congruences between the specializations, up to finitely many coefficients, as we next recall.

\begin{defn}Let $f\in S_k(N,\chi)$ and $g\in  S_r(M,\psi)$, $K$ be a number field which contains both definition fields of $f$ and $g$ and $\frak{p}$ be a prime in $K$. We denote
$$
f\equiv g\pmod{\frak{p}}
$$
if for all but finitely many Fourier coefficients $c(n,f)$ and $c(n,g)$, it holds
$$
c(n,f)\equiv c(n,g)\pmod{\frak{p}}.
$$
\end{defn}

Let us denote by $\omega$ the Teichm\"uller character modulo $l-1$. With the help of Hida families one can prove the following result, which is stated in a slightly different manner in \cite{ghate} Theorem 4:

\begin{prop}Let $f\in  S_k(N,\chi)$ a normalized eigenform for the full Hecke algebra which is a newform (i.e. a primitive cusp form). There are $f_r\in S_r(N, \chi\omega^{k-r})$, with $r\geq 2$, such that $f_r^{(l)}$ is the weight $r$ specialisation of a Hida family $(\Lambda_{\bold{f}},\Omega_{\bold{f}},\Omega_{\bold{f}}^{cl},\bold{f}(q))$ passing by $f$ and such that
\begin{itemize}
\item  $f_k^{(l)}=f^{(l)}$,
\item $f_r^{(l)}$ is a normalized eigenform of level $Nl$ for each $r$
\item $f_r$ is $l$-ordinary for each $r$ (and so is $f_r^{(l)}$),
\item $f_{r_1}^{(l)}\equiv f_{r_2}^{(l)}\pmod{\lambda}$ for each $r_1,r_2\geq 2$.
\end{itemize}
Here $\lambda$ is a prime over $l$ in $K_{r_1,r_2}$, a number field containing both fields of definitions of $f_{r_1}$ and $f_{r_2}$.
\label{congruence}
\end{prop}
\begin{proof}From Definition \ref{defnHida} and Theorem \ref{hidapaso}, we start by considering a Hida family $(\Lambda_{\bold{f}},\Omega_{\bold{f}},\Omega_{\bold{f}}^{cl},\bold{f}(q))$ passing by $f$, where $\bold{f}(q)=\sum_{n\geq 1}\bold{a}_nq^n\in\ \Lambda_{\bold{f}}[[q]]$ and where $\Lambda_{\bold{f}}$ is a finite and flat extension of $\Lambda$. Since $\Lambda$ is local, it follows that $\Lambda_{\bold{f}}$ is $\Lambda$-free, namely, $\Lambda_{\bold{f}}\cong \Lambda\theta_1\oplus...\Lambda\theta_d$ as modules. On the other hand, since $\Lambda\subset \Lambda_{\bold{f}}$ is a finite ring extension, then it is an integral extension and hence, by a theorem by Abhyankar (\cite{aby} Theorem 3), we can assume that for all $1\leq i\leq d$, the element $\theta_i$ is a power series in $x^{1/t}$ for some $t\geq 1$ with coefficients over $\overline{\mathbb{Q}}_l$. Further, recall that the weight $r$ specialization on $\Lambda_{\bold{f}}$  is a ring homomorphism whose restriction to $\Lambda$ is the classical character $\chi_r$, which is defined via $x\mapsto \varepsilon(u)u^r-1$ for $u$ a topological generator of $\Gamma$ and $\varepsilon$ a finite order character of $\Gamma$ (see Theorem I in \cite{hidagalois}). 

Hence, the $n$-th coefficient $\bold{a}_n$ of $\bold{f}(q)$ has the form $\sum_{i=1}^da_{n,i}\theta_i$ with $a_{n,i}\in\Lambda$ and $\theta_i\in\overline{\mathbb{Q}}_l[[x^{1/t}]]$.  Let us write $f_r^{(l)}=\sum_{n\geq 1}a_n(r)q^n$ for the weight $r$ specialization of $\bold{f}(q)$.  Then, for all $r_1, r_2\geq 2$, we have that $a_n(r_1)\equiv a_n(r_2)\pmod{\mathcal{L}}$ where $\mathcal{L}$ is the maximal ideal in $\mathcal{O}_{\mathbb{C}_l}$, the valuation ring of $\mathbb{C}_l$. But since the $r_1$ and $r_2$ specializations of the Hida family are defined over a common number field, say, $K_{r_1,r_2}$, the congruence is modulo $\mathcal{L}\cap K_{r_1,r_2}$, generated by a prime element $\lambda$.

\end{proof}

Now we can prove our main theorem in the ordinary case.

\begin{thm}For our $l$-ordinary cusp form $f$, there exists an infinite sequence $\{k_r\}_{r\geq 1}\subseteq\mathbb{Z}$, and for each $k_r$, a cusp form $f_{k_r}\in S_{k_r}(N,\chi)$  such that $\rho_{f_{k_r,l}}:G_{\mathbb{Q}}\to\mathrm{GL}_2(\mathcal{O}_{f_{k_r},\lambda})$ is a lift of $\overline{\rho}_{f,l}$ and such that $\rho_{f_{k_r,l}}|_{\mathbb{Q}_l}$ is
\begin{itemize}
\item[1.] crystalline of Hodge-Tate weight $\{0,k_r-1\}$,
\item[2.] potentially diagonalizable.
\end{itemize}
\label{ordcase}
\end{thm}
\begin{proof}
Consider the Hida family $(\Lambda_f,\Omega_f,\Omega_f^{cl},\bold{f})$ passing by $f$ in weight $k$ and consider the infinite subset $\Omega_{f,k}^{cl}=\{k_r\equiv k\pmod{l-1}\}$, so that for each $k_r\in \Omega_{f,k}^{cl}$, the specialization $f^{(l)}_{k_r}$ is the ordinary $l$-stabilisation of an ordinary Hecke-eigenform $f_{k_r}\in S_{k_r}(N,\chi)$ of weight $k_r$.

Denote by $\rho_{f_{k_r,l}}$ the $l$-adic representation attached to $f_{k_r}$ by Deligne. Let $K_{f_r,f}$ be the compositum of $K_f$ and $K_{f_r}$ and choose a prime $\lambda$ of $K_{fr,f}$ over $l$. We have $f_{k_r}^{(l)}\equiv f^{(l)}\pmod{\lambda}$ by Prop. \ref{congruence}. From this we conclude that $\overline{\rho}_{f_{k_r,l}}\cong\overline{\rho}_{f}$:

Indeed: first we observe that the $n$-th Fourier coefficient of an ordinary Hecke-eigenform coincides with that of its $l$-stabilisation when $l\nmid n$. Second, the isomorphism class of a semisimple representation of $G_{\mathbb{Q}}$ into $GL_2(\mathbb{F}_l)$ is determined by the traces and determinants of the Frobenius elements $\phi_p$ for $p$ prime outside a finite set: namely, those primes not dividing $Nl$, with $N$ the Artin conductor.

Since for each prime $p\nmid lN$, we have $c(f_{k_r}^{(l)},p)=c(f_{k_r},p)$ and $c(f^{(l)},p)=c(f,p)$ and also for all $p\neq l$ we have $c(f_{k_r}^{(l)},p)\equiv c(f^{(l)},p)\pmod{\lambda}$, it also follows that $\overline{\rho}_{f_{k_r,l}}(\phi_p)=\overline{\rho}_{f}(\phi_p)$, as claimed.

To conclude, since $l\nmid N$, $f_{k_r}$ is crystalline (and ordinary) for each $r\geq 2$, then $\rho_{f_{k_r,l}}$ is potentially diagonalizable.
\end{proof}

\section{The non-ordinary case} Let now $f\in S_k(\Gamma_0(N),\chi)$ be non-ordinary. Since we are assuming that $0\leq k\leq l-1$ and since $\mathbb{Q}_l$ is totally unramified and since $l\nmid N$, $\rho_{f,l}|_{\mathbb{Q}_l}$ is crystalline, hence by Theorem \ref{thmbase} b), $\rho_{f,l}|_{\mathbb{Q}_l}$ is potentially diagonalizable.

If we try to repeat the argument of the proof of Theorem \ref{ordcase} replacing Hida families by Coleman families in our non-ordinary case, we find that each specialization $f_{k_r}$ is neither ordinary nor its weight is in the Fonaine-Lafaille range, as it increases with $r$. In this section we deal with the non-ordinary case with an alternative approach independent (at first sight!) on Coleman families.

\subsection{Local lifts}

Since $f$ is non-ordinary at $l$, the reduction of its $l$-th Fourier coefficient vanishes in $\overline{\mathbb{F}}_l$ hence, by a theorem of Fontaine (a proof of which can be found in \cite{edixhoven}, Section 6.8), we have
$$
\overline{\rho}_{f,l}|_{I_l}=\left(\begin{array}{cc}\psi_2^{k-1} & 0\\ 0 & \psi_2^{(k-1)l}\end{array}\right),
$$
where $\psi_2$ is a fundamental character of level $2$. Notice that $\psi_2^{(k-1)l}$ is conjugated to $\psi_2^{k-1}$ and the exponents of the diagonal entries have the form $al+b$ and $a+bl$, where $\{a,b\}=\{0,k-1\}$, the Hodge-Tate weights of $f$.

\begin{lem}For any $k_r\equiv k\pmod{l^2-1}$ there exists a potentially diagonalizable crystalline lift $\rho_{k_r,l}:G_{\mathbb{Q}_l}\to\mathrm{GL}_2(\overline{\mathbb{Q}}_l)$ of $\overline{\rho}_{f,l}|_{G_{\mathbb{Q}_l}}$ of Hodge-Tate weight $\{0,k_r-1\}$ and determinant $\chi\epsilon^{k_r-1}$.
\label{firstlift}
\end{lem}
\begin{proof}Denote by $\mathbb{Q}_{l^2}$ the unique unramified extension of $\mathbb{Q}_l$ of degree $2$. We apply \cite{blgg} Lemma 4.1.19 to our setting. First, in the notations of loc. cit. we have $e=1$, $k'=k$, $J$ has just one embedding $\sigma$, $J^c$ is empty and $\delta_{\sigma}=0$.  

Hence, invoking \cite{blgg} Lemma 4.1.19, we obtain a crystalline and potentially diagonalisable lift $\rho_{k_r,l}:G_{\mathbb{Q}_l}\to\mathrm{GL}_2(\overline{\mathbb{Q}}_l)$ of $\overline{\rho}_{f,l}$ such that
$$
\rho_{k_r,l}|G_{\mathbb{Q}_{l^2}}=\left(
\begin{array}{cc}
\varepsilon\psi_{2,k_r}^{k_r-1} & 0\\
0 & \varepsilon'\psi_{2,k_r}^{(k_r-1)l}
\end{array}
\right),
$$
where $\varepsilon,\varepsilon':G_{\mathbb{Q}_l}\to\overline{\mathbb{Q}}_l^{*}$ are unramified characters and $\psi_{2,k_r}$ a crystalline lift of $\psi_2$, with Hodge-Tate weight $\{0,k_r-1\}$. The characters $\varepsilon$ and $\varepsilon'$ are given by Lemma 4.1.15 of loc. cit., as the crystalline character lifting $\psi_2^{k-1}$ mentioned there is unique up to unramified twist. 

In particular, $\det(\rho_{k_r,l})$ is a crystalline lift of $\det(\overline{\rho}_{f,l})=\chi\epsilon_l^{k-1}$ (denoting likewise by $\chi$ the Nebentype character and its reduction mod $l$). In particular,  $\det(\rho_{k_r,l})=\chi^*\epsilon_l^{k_r-1}$, with $\chi^*$ an unramified lift of $\chi$ (modulo $l$). Notice that the character $\chi$ (seen as character of $\mathrm{Gal}(\overline{\mathbb{Q}}_l/\mathbb{Q}_l)$) is unramified. Denote still by $\rho_{k_r,l}$ the twisted representation $\alpha\otimes \rho_{k_r,l}$ with
$\alpha=\sqrt{\chi(\chi^*)^{-1}}$, for a fixed choice of square root.

Notice that since $\chi$ and $\chi^*$ are unramified so is $\alpha$:  indeed, identifying characters of $\mathrm{Gal}(\mathbb{Q}_l^{unr}/\mathbb{Q}_l)$ with characters of $\mathrm{Gal}(\overline{\mathbb{F}}_l/\mathbb{F}_l)$, we see $\chi(\chi^*)^{-1}$ is defined by its image on a topological cyclic generator $\theta$ of $\mathrm{Gal}(\overline{\mathbb{F}}_l/\mathbb{F}_l)$ and hence so is $\alpha$.

Finally, since $\alpha\equiv Id\pmod{l}$ and $\alpha$ is unramified (hence crystalline),  $\rho_{k_r,l}$ is also a crystalline lift (and potentially diagonalizable) of $\overline{\rho}_{f,l}$ with the right determinant.
\end{proof}

\subsection{From local to global} From the family $\rho_{k_r,l}$ constructed in Lemma \ref{firstlift}, our final goal is to produce a family of modular and crystalline lifts of $\overline{\rho}_{f,l}$. This involves the definition of a suitable global deformation ring and a positive lower bound for its dimension, which we carry out in this and next sections. We follow very closely the approach of \cite{kwannals} Section 4.

First, we start by defining the following family of local rings for a finite set of rational primes attached to several deformation conditions. In the rest of this section, unless we state otherwise, $\mathcal{O}$ will denote the ring of integers of $K_{f,\lambda}$. From now on, we set $\mu_r=\chi\epsilon^{k_r-1}$. Notice that the residual determinant is $\chi\epsilon^{k-1}$.

\begin{defn}Fix $k_r$. Let $S=\left\{\nu\mid N, \nu\mbox{ prime}\right\}\cup\left\{l,\infty\right\}$ and for each $\nu\in S$, define the following ring $R^{\square,\mu_r}_{\mathcal{O},\nu}$:
\begin{itemize}
\item for $\nu=l $,  define $R^{\square, \mu_r}_{\mathcal{O},l}$ to be the framed universal crystalline deformation ring of Hodge-tate weight $\{0,k_r-1\}$ for $\overline{\rho}_{f,l}$ with determinant $\mu_r$. 
\item for $\nu=\infty$, define $R^{\square, \mu_r}_{\mathcal{O},\infty}$ to be framed universal deformation ring for $\overline{\rho}_{f,l}$ corresponding to odd deformations of $\overline{\rho}_{f,l}$ with determinant $\mu_r$. 
\item for $\nu\neq l,\infty$, if $\overline{\rho_{f,l}}$ is not the twist of a semistable representation, define $R^{\square, \mu_r}_{\mathcal{O},\nu}$ to be the framed universal deformation ring corresponding to inertia-rigid deformations, i.e., deformations $\rho$ with $\rho(I_{\nu})$ finite and determinant $\mu_r$. 
\item finally, for $\nu\neq l,\infty$,  if $\overline{\rho_{f,l}}$ is the twist of a semistable representation, define $R^{\square, \mu_r}_{\mathcal{O},\nu}$ to be the framed universal deformation ring corresponding to semistable lifts and determinant $\mu_r$.
\end{itemize}
\label{defnlocal}
\end{defn}


These rings satisfy the following properties.

\begin{prop}The rings defined in \ref{defnlocal} are $\mathcal{O}$-flat domains of finite dimension over $\mathcal{O}$. In particular:
\begin{itemize}
\item[1. ] $dim_{\mathcal{O}}(R^{\square, \mu_r}_{\mathcal{O},l})=4$,
\item[2. ] for $\nu\neq l,\infty$, $dim_{\mathcal{O}}(R^{\square, \mu_r}_{\mathcal{O},\nu})=3$ (inertia-rigid and semistable cases), 
\item[3. ] $dim_{\mathcal{O}}(R^{\square, \mu_r}_{\mathcal{O},\infty})=2$.
\end{itemize}
\end{prop}
\begin{proof}The existence of $R^{\square, \mu_r}_{\mathcal{O},l}$ and its dimension are justified as follows:  by Corollary 3.3.3 in \cite{allen}, the ring $R^{\square, \mu_r}_{\mathcal{O},l}$ is $\mathcal{O}$-flat, reduced, and if non-zero, the ring $R^{\square, \mu_r}_{\mathcal{O},l}[1/l]$ is equidimensional of Krull dimension $4$. The ring $R^{\square, \mu_r}_{\mathcal{O},l}$ also depends on the inertial type $\tau=\rho_{k_r,l}|_{I_l}$ (where we denote by $\rho_{k_r,l}$, by abuse of notation, the Weil-Deligne representation associated to the lift obtained in Lemma 3.1), but to ease notation we will omit the reference to $\tau$. The fact that this ring is non-zero is immediate since there is at least the point corresponding to the mentioned lift constructed in Lemma 3.1.

Now, the argument of \cite{eg} Lemma 4.3.1 shows that $R^{\square}_{\mathcal{O},l}$, the ring without fixed determinant, is a power series ring in one variable over $R^{\square,\mu_r}_{\mathcal{O},l}$. Hence $R^{\square, \mu_r}_{\mathcal{O},l}$ is $\mathcal{O}$-flat if and only if $R^{\square}_{\mathcal{O},l}$ is $\mathcal{O}$-flat. Hence both are  $\mathcal{O}$-flat and we have: 
$$
dim_{\mathcal{O}}(R^{\square, \mu_r}_{\mathcal{O},l})=dim(R^{\square, \mu_r}_{\mathcal{O},l})-dim(\mathcal{O})=dim(R^{\square, \mu_r}_{\mathcal{O},l}[1/l])=4.
$$

For 2 and 3, these rings are the same as in \cite{kwannals} Theorem 3.1, where the $\mathcal{O}$-flatness and the dimensions are given. The ring $R^{\square}_{\mathcal{O},\infty}$ is non-zero due to \cite{kwannals} Theorem 3.3. As for $R^{\square}_{\mathcal{O},\nu}$, for $\nu\neq l,\infty$, in the inertia-rigid case, the existence the ring and the existence of points defined over possibly a finite extension of $\mathcal{O}$ is ensured by \cite{kwannals} Sections 3.3.1 to 3.3.3, and in the semistable representation twist case, it is granted by \cite{kwannals} Section 3.3.4.
\end{proof}


\begin{notation}
Now, slightly abusing notation we will denote still by $R^{\square, \mu_r}_{\mathcal{O},l}$ the connected component of the framed universal crystalline deformation ring of Hodge-tate weight $\{0,k_r-1\}$ for $\overline{\rho}_{f,l}$ with determinant $\mu_r$   corresponding to the potentially diagonalizable lift $\rho_{k_r,l}$ proved in lemma \ref{firstlift}. This is still an $\mathcal{O}$-flat domain of relative dimension $4$.
\end{notation}
\begin{prop}The ring $R^{\square, loc, \mu_r}_S:=\displaystyle\widehat{\otimes}_{\nu\in S}R_{\mathcal{O},\nu}^{\square, \mu_r}$ is (after possibly replacing $\mathcal{O}$ by the ring of integers of a finite extension of $\mathbb{Q}_p$), flat over $\mathcal{O}$ of relative dimension $3|S|$.
\label{dim1}
\end{prop}
\begin{proof}The ring $R^{\square, \mu_r}_{\mathcal{O},l}$ contributes $4=3|S_l|+1$ to the tensor product dimension ($S_l$ denotes the set of primes above $l$, in our case just $l$ itself). The ring $R^{\square, \mu_r}_{\mathcal{O},l}$ contributes $2$ and the rest contributes $3|S\setminus\{l,\infty\}|$. Summing all these numbers gives the result.
\end{proof}
Define 
\begin{equation}
\widehat{R}_S^{\square, loc, \mu_r}:=\widehat{\otimes}_{\nu\in S} R_{\nu}^{\square, \mu_r},
\end{equation}
where $R_{\nu}^{\square}$ is the usual framed deformation ring (with no conditions except fixed determinant). 
\begin{prop}In a natural way $R^{\square, \mu_r}_{\mathcal{O},S}$ is an $\widehat{R}^{\square, loc, \mu_r}_S$-algebra.
\end{prop}
\begin{proof}The ring $R^{\square, \mu_r}_{\mathcal{O},S}$ represents the functor sending $A$ to the set of isomorphism classes of pairs $ (\rho_A,\{\beta_{A,\nu}\}_{\nu\in S})$ with $\det(\rho_A)=\mu_r$ and for each $\nu\in S$, the ring $R_{\nu}^{\square, \mu_r}$ represents the functor $A$ to the set of isomorphism classes of pairs $(\rho_A,\beta_{A,\nu})$, where $\rho_A$ is a lift of $\overline{\rho}$ to $A$, $\beta_{A,\nu}$ is a basis lifting the prescribed residual basis of $\overline{\rho}$ and $\det(\rho_A|_{D_{\nu}})=\mu_r|_{D_{\nu}}$. The forgetul map $(\rho_A,\{\beta_{A,\nu}\}_{\nu\in S})\mapsto (\rho_A,\beta_{A,\nu})$ gives an arrow $R_{\nu}^{\square}\to R^{\square, \mu_r}_{\mathcal{O},S}$, hence an arrow $\widehat{R}^{\square, loc, \mu_r}_S\to R^{\square, \mu_r}_{\mathcal{O},S}$.
\end{proof}

Now, define 
\begin{equation}
\widehat{R}^{\square, \mu_r}_S=R^{\square, \mu_r}_{\mathcal{O},S}\widehat{\otimes}_{\widehat{R}_S^{\square,loc, \mu_r}}R^{\square,loc, \mu_r}_S.
\end{equation}

\begin{defn}Let $\mathcal{R}_{k_r,global}$ be the image of the usual unframed deformation ring $R_S^{\mu_r}$ in $\widehat{R}^{\square, \mu_r}_S$ via the natural map $R_S^{\mu_r}\to R^{\square, \mu_r}_{\mathcal{O},S} $.
\end{defn}

The ring $\mathcal{R}_{k_r,global}$, represents the functor that assigns to a CNL-$\mathcal{O}$-algebra $A$, the set of triples $\{(\rho_A,(\rho_{\nu})_{\nu\in S},(g_{\nu})_{\nu\in S})\}$ where $\rho_A$  (resp. $\rho_{\nu}$) is a lift of $\overline{\rho}$ (resp. $\overline{\rho}|_{D_{\nu}}$) to $A$ and such that for each $\nu\in S$, $g_{\nu}\in GL_2(A)_1$, $\rho_{\nu}=int(g_{\nu})\rho_{A}|_{D_{\nu}}$, and for each $\nu\in S$,  the local representation $\rho_A|D_{\nu}\cong \rho_{\nu}: D_{\nu}\to GL_2(A)$ factors through an arrow $R_{\nu,\mathcal{O}}^{\square, \mu_r}\to A$, which is equivalent to saying that they satisfy each imposed local condition at $\nu$. In particular, $\rho_A|D_l$ is crystalline and potentially diagonalizable of Hodge-Tate weight $\{0,k_r-1\}$ in the sense of \cite{gee} pag. 29.

Two tuples $(\rho_A, \{\rho_{\nu}\}_{\nu\in S}, \{g_{\nu}\}_{\nu\in S})$ and $(\rho'_A, \{\rho'_{\nu}\}_{\nu\in S}, \{g'_{\nu}\}_{\nu\in S})$ are equivalent if $(\rho'_A, \{\rho'_{\nu}\}_{\nu\in S}, \{g'_{\nu}\}_{\nu\in S})=(int(g)\rho_A, \{\rho_{\nu}\}_{\nu\in S}, \{g_{\nu}g^{-1}\}_{\nu\in S})$ for some $g\in \mathrm{GL}_2(A)_1$.

Notice that the ring $\mathcal{R}_{k_r,global}$ has been constructed following the recipe given by \cite{kwannals} in Section 4.1.1. In our process, all the rings are the same as in \cite{kwannals} except the local ring $R^{\square, \mu_r}_{\mathcal{O},l}$, giving the local deformation condition at $l$.

To state the next ingredient in our proof, we need to recall the concept of torsor. We refer the reader to \cite{kwannals} Section 2.4 for details on quotients of functors by group actions.

\begin{defn}Let $X$ be a representable functor from the category of complete noetherial local $\mathcal{O}$-algebras to sets and let $G$ be a smooth group acting freely on $X$. Let $O$ be the functor of orbits of $G$ acting on $X$ and denote $G_{O}:=G\times O$. We say that $X$ is a torsor over $O$ of group $G_{O}$ if the natural map $G_{O}\times X\to X\times_{O} X$ is an isomorphism.
\end{defn}

For a complete noetherian local $\mathcal{O}$-algebra $R$, the formal spectrum $\mathrm{Specf}(R)$ defines a representable functor by $\mathrm{Specf}(R)(S)=\mathrm{Hom}_{\mathcal{O}}(S,A)$.

\begin{prop}[\cite{kwannals} Prop. 4.1]The ring $\widehat{R}^{\square, \mu_r}_S$ is a power series ring over $\mathcal{R}_{k_r,global}$ in $4|S|-1$ variables.
\label{dim2}
\end{prop}
\begin{proof}This is immediate from the fact that $\mathrm{Specf}(\widehat{R}^{\square, \mu_r}_S)$ is a $\mathrm{Specf}(\mathcal{R}_{k_r,global})$-torsor of group $(\prod_{\nu\in S}(\mathrm{GL}_2)_1)/\mathbb{G}_m$. The reader is referred to \cite{kwannals} Prop 4.1. for a proof, which is independent of what are the precise deformation conditions.
\end{proof}

\subsection{The dimension of the global deformation ring.} In this section we prove that the ring $\mathcal{R}_{k_r,global}$ has positive absolute dimension. Our proof follows very closely Proposition 4.5 in \cite{kwannals}. Actually, our situation is a very particular case of the general framework considered therein. Next we recall the relevant facts of Galois cohomology presented in Section 4 of \cite{kwannals} but adapted to our case.

Let $F$ be a number field and $S$ a finite set of places of $F$ possibily containing the archimedean ones. Let $M$ be a $G_F$-module where the action is unramified outside $S$. For $k=0,1$, denote by $H^k(S,M)$ the cohomology group $H^k(G_{F,S},M)$. If for each $\nu\in S$ we are given a subspace $L_{\nu}$ of $H^k(D_{\nu},M)$, we denote by $H^k_{\{L_{\nu}\}}(S,M)$ the preimage of the subspace $\prod_{\nu\in S}L_{\nu}\subseteq \prod_{\nu\in S}H^k(D_{\nu},M)$ under the restriction map $H^k(S,M)\to\prod_{\nu\in S}H^k(D_{\nu},M)$. Following \cite{kwannals} 4.1.1, we fix $k=1$, and consider two situations:

\begin{itemize}
\item Case 1: $M=Ad^0(\overline{\rho}_{f,l})$ and for each $\nu\in S$, $L_{\nu}$ is the image of $H^0(D_{\nu}, Ad/Ad^0)$ in $H^1(D_{\nu},Ad^0)$, which has dimension $0$, since we are assuming $l>6$.
\item Case 2: $M=Ad(\overline{\rho}_{f,l})$, and for each $\nu\in S$, $L_{\nu}=0$. Notice that for us, the set $V$ considered in 4.1.1 of \cite{kwannals} is empty.
\end{itemize}

Since in our case $l>6$, we have $Ad(\overline{\rho}_{f,l})=Ad^0(\overline{\rho}_{f,l})\oplus Z$ with $Z$ the scalar matrices in $M_2(\mathbb{F}_l)$. Write $Ad$ and $Ad^0$ for $Ad(\overline{\rho}_{f,l})$ and $Ad^0(\overline{\rho}_{f,l})$ respectively, define $(Ad^0)^*:=\mathrm{Hom}_{\mathbb{F}}(Ad^0,\mathbb{F})$ and $(Ad^0)^*(1)=\mathrm{Hom}_{\mathbb{F}}(Ad^0,\mu_l^*)$.

\begin{lem}[\cite{kwannals} Lemma 4.3] Consider the exact sequence
$$
0\to H^0(S,Ad^0)\to H^0(S,Ad)\to \mathbb{F}\to H^1(S,Ad^0)\to H^1(S,Ad).
$$
Denote by $H^1(S,Ad^0)^{\eta}$ and $H^1_{\{L_{\nu}\}}(S,Ad^0)^{\eta}$ the images of $H^1(S,Ad^0)\to H^1(S,Ad)$ and $H^1_{\{L_{\nu}\}}(S,Ad^0)\to H^1(S,Ad)$. Then:
\begin{itemize}
\item[1. ] The surjective maps $H^1(S,Ad^0)\to H^1(S,Ad^0)^{\eta}$ and $H^1_{\{L_{\nu}\}}(S,Ad^0)\to H^1_{\{L_{\nu}\}}(S,Ad^0)^{\eta}$ are isomorphisms.
\item[2. ] $H^0(F,Ad^0)=H^0(F, (Ad^0)^{*}(1))=0$.
\item[3. ] There is an injection $H^1(F,Z)\hookrightarrow H^1(F,Ad)$.
\label{lema43}
\end{itemize}
\end{lem}

We will make use of the following result, which is also a particular case of Lemma 4.4 of \cite{kwannals}. We provide the proof in our setting for the sake of completeness.

\begin{prop}The minimal number of generators of $R^{\square, \mu_r}_{\mathcal{O},S}$ (analogously $\widehat{R}^{\square, \mu_r}_S$) over $\widehat{R}_S^{\square,loc, \mu_r}$ (analogously $R^{\square,loc, \mu_r}_S$) is 
$$
g:=dim_{\mathbb{F}}(H^1_{\{L_{\nu}\}}(S,Ad^0))+\sum_{\nu\in S}dim_{\mathbb{F}}(H^0(D_{\nu},Ad))-dim_{\mathbb{F}}(H^0(F,Ad)).
$$
\label{ngen}
\end{prop}
\begin{proof}Denote by $\frak{m}$ the maximal ideal of $R^{\square, \mu_r}_{\mathcal{O},S}$. It is enough to prove that the dimension of the relative tangent space $T:=\mathrm{Hom}_{\mathcal{O}}(\frak{m}/\frak{m}^2\otimes_{\widehat{R}_S^{\square,loc,\mu_r}}\mathbb{F},\mathbb{F})$ is $g$.

First, notice that an element in $T$ corresponds to an infinitesimal deformation $\rho$ of $\overline{\rho}_{f,l}$ to $GL_2(\mathbb{F}[\varepsilon])$ with fixed determinant $\mu_r$, together with $\{\beta_{\nu}\}_{\nu\in S}$ such that for each $\nu\in S$, $\beta_{\nu}$ lifts the prescribed basis $\beta$ of $\overline{\rho}_{f,l}$ and $(\rho|_{D_{\nu}},\beta_{\nu})\cong \overline{\rho}_{f,l}\otimes_{\mathbb{F}}\mathbb{F}[\varepsilon]$. The space of  such deformations is given by  $H^1_{\{L_{\nu}\}}(S,Ad^0)^{\eta}$, whose dimension is $dim_{\mathbb{F}}(H^1_{\{L_{\nu}\}}(S,Ad^0))$ by Lemma \ref{lema43}.

Now, we argue as in Lemma 3.2.2 of \cite{kisinannals}: given an infinitesimal deformation $\rho$  as in the previous paragraph, the space of possible choices for the basis $\{\beta_{\nu}\}$ is given by $H^0(F_{\nu},Ad)$. Finally, two sets of choices of bases $\{\beta_{\nu}\}_{\nu\in S}$ and $\{\beta'_{\nu}\}_{\nu\in S}$ are equivalent if there exists $\phi\in GL_2(\mathbb{F}[\varepsilon])$, commuting with the $D_{\nu}$ actions which reduce to a homothety on $\mathbb{F}^2$ and brings $\beta_{\nu}$ to $\beta'_{\nu}$ for each $\nu$, hence the result follows.
\end{proof}

\begin{lem}[Wiles, see \cite{kwannals} 4.1.16] With notations as above, it holds:
\begin{equation}
\frac{|H^1_{\{L_{\nu}\}}(S,Ad^0)|}{|H^1_{\{L_{\nu}^{*}\}}(S,(Ad^0)^*(1))}=\frac{|H^0(F,Ad^0)|}{|H^0(F,(Ad^0)^*(1))}\prod_{\nu\in S}\frac{1}{H^0(D_{\nu},Ad^0)}.
\end{equation}
\label{wiles}
\end{lem}
\begin{proof}This is Equation (2) in Section 4.1.6 of \cite{kwannals}, with $V=\emptyset$ and $l>2$, as in our case.
\end{proof}

Next we prove the technical core of this paper.

\begin{thm}Notations as before, it holds
$$
dim(\mathcal{R}_{k_r, global})\geq 1.
$$
\label{cotadimension}
\end{thm}
\begin{proof}
Let $g=dim_{\mathbb{F}}(H^1_{\{L_{\nu}\}}(S,Ad^0))+\sum_{\nu\in S}dim_{\mathbb{F}}(H^0(D_{\nu},Ad))-dim_{\mathbb{F}}(H^0(F,Ad))$ be, as in Lemma \ref{ngen}, the minimal number of generators of $R^{\square}_{\mathcal{O},S}$ over $\widehat{R}_S^{\square,loc}$. By Wiles's Lemma \ref{wiles} and Lemma \ref{lema43} (2), this can be written as
$$
dim_{\mathbb{F}}(H^1_{\{L_{\nu}^*\}}(S,(Ad^0)^*(1)))-dim_{\mathbb{F}}(H^0(F,Ad))+\sum_{\nu\in S}(dim_{\mathbb{F}}(H^0(D_{\nu},Ad))-dim(H^0(D_{\nu,Ad^0}))).
$$
Considering the exact sequence
$$
0\to H^0(D_{\nu},Ad^0)\to H^0(D_{\nu},Ad)\to\mathbb{F}\to 0,
$$
we have
\begin{equation}
g=dim_{\mathbb{F}}(H^1_{\{L_{\nu}^*\}}(S,(Ad^0)^*(1)))+|S|-1.
\label{generadores}
\end{equation}
Hence, we have a presentation
$$
R^{\square, \mu_r}_{\mathcal{O},S}\cong \widehat{R}_S^{\square,loc, \mu_r}[[X_1,...,X_g]]/J 
$$
which induces an isomorphism on the relative to $\widehat{R}_S^{\square,loc, \mu_r}$ tangent spaces of $\widehat{R}_S^{\square,loc, \mu_r}[[X_1,...,X_g]]$ and $R^{\square, \mu_r}_{\mathcal{O},S}$, since the dimension of both tangent spaces is $g$, as seen in the proof of Proposition \ref{ngen}. Denote by $r(J)$ the minimal number of generators of $J$, so that $r(J)=dim(J/\frak{m}J)$, where $\frak{m}$ is the maximal ideal of $\widehat{R}_S^{\square,loc, \mu_r}[[X_1,...,X_g]]$. 

Next, as we will prove apart in Lemma \ref{techlemma}, we observe that $r(J)\leq dim_{\mathbb{F}}(H^1_{\{L_{\nu}^*\}}(S,(Ad^0)^*(1)))$. And now, the rest of the proof is exactly as in Prop. 4.5 of \cite{kwannals}: 

We start with the presentation
$$
R^{\square, \mu_r}_{\mathcal{O},S}\cong \widehat{R}_S^{\square,loc, \mu_r}[[X_1,...,X_g]]/J.
$$
with it and the natural maps $\widehat{R}_S^{\square,loc, \mu_r}\to R_S^{\square,loc, \mu_r}$ and $R_{\mathcal{O},S}^{\square, \mu_r}\to \widehat{R}_S^{\square, \mu_r}$ we deduce a presentation
$$
\widehat{R}_S^{\square, \mu_r}\cong  R_S^{\square,loc, \mu_r}[[X_1,...X_g]]/J',
$$
where $J'$ is generated by at most $dim_{\mathbb{F}}(H^1_{\{L_{\nu}^*\}}(S,(Ad^0)^*(1)))$ elements.

Hence, denoting by dim the absolute dimension of a ring, we have:
$$
dim(\widehat{R}_S^{\square, \mu_r})\geq dim(R_S^{\square,loc, \mu_r})+g-dim_{\mathbb{F}}(H^1_{\{L_{\nu}^*\}}(S,(Ad^0)^*(1))), 
$$
which due to Equation \ref{generadores} is lower bounded by $dim(R_S^{\square,loc, \mu_r})+|S|-1$, which equals $4|S|$, by Proposition \ref{dim1} ($\mathcal{O}$ has absolute dimension $1$).

To conclude the proof, we apply Poposition \ref{dim2}:
$$
dim(\widehat{R}_S^{\square, \mu_r})=dim(\mathcal{R}_{k_r, global})+4|S|-1\geq 4|S|.
$$
and the result holds.
\end{proof}

We are now left to prove the following technical result:

\begin{lem}Notations as before, it holds:
$$
r(J)\leq dim_{\mathbb{F}}(H^1_{\{L_{\nu}^*\}}(S,(Ad^0)^*(1))).
$$
\label{techlemma}
\end{lem}
\begin{proof}The strategy is to define an injective $\mathbb{F}$-linear map $f:Hom_{\mathbb{F}}(J/\frak{m}J,\mathbb{F})\to H^1_{\{ L_{\nu}^*\}}(S,(Ad^0)^*(1))^*$. For this, we will define an $\mathbb{F}$-bilinear pairing
$$
(\phantom{d},\phantom{d}):H^1_{\{ L_{\nu}^*\}}(S,(Ad^0)^*(1))\times Hom_{\mathbb{F}}(J/\frak{m}J,\mathbb{F})\to\mathbb{F}
$$
and check that $(\phantom{d},\phantom{d})$ is non-degenerate on the right.

So, let $u\in Hom_{\mathbb{F}}(J/\frak{m}J,\mathbb{F}) $ and $[x]\in H^1_{\{ L_{\nu}^*\}}(S,(Ad^0)^*(1))$. We have the exact sequence
\begin{equation}
0\to J/\frak{m}J \to \widehat{R}_S^{\square,loc, \mu_r}[[X_1,...,X_g]]/\frak{m}J \to R^{\square,  \mu_r}_{\mathcal{O},S}\to 0,
\label{eqaux}
\end{equation}
which after push-forward by $u$ gives
\begin{equation}
0\to I_u \to R_u \to R^{\square,  \mu_r}_{\mathcal{O},S}\to 0,
\label{eqaux2}
\end{equation}
where $R_u$ is an $\widehat{R}_S^{\square,loc,  \mu_r}$-algebra, via 
$$
\widehat{R}_S^{\square,loc,  \mu_r}\hookrightarrow \widehat{R}_S^{\square,loc,  \mu_r}[[X_1,...,X_g]]/\frak{m}J\to u_{*}\widehat{R}_S^{\square,loc,  \mu_r}[[X_1,...,X_g]]/\frak{m}J=R_u.
$$

Let $(\rho,(\rho_{\nu})_{\nu\in S}, (g_{\nu})_{\nu\in S})$ be the tuple representing the tautological point of $R^{\square,  \mu_r}_{\mathcal{O},S}$, i.e., corresponding to the identity on $R^{\square,  \mu_r}_{\mathcal{O},S}$. Since $R_u$ is an $\widehat{R}_S^{\square,loc,  \mu_r}$-algebra, then for all $\nu\in S$ we obtain a lift $\tilde{\rho}_{\nu}$ of $\rho_{\nu}$ with values in $GL_2(R_u)$. Likewise, choose lifts $\tilde{g}_{\nu}\in GL_2(R_u)$ of the $g_{\nu}$ and write $\tilde{\rho}_{\nu}':=int(\tilde{g}_{\nu}^{-1})(\tilde{\rho}_{\nu})$.

Consider a set-theoretic lift $\tilde{\rho}:G_S\to GL_2(R_u)$ of the universal representation $\rho_S^{\square,  \mu_r}:G_S\to GL_2(R^{\square,  \mu_r}_S)$ such that the image of $\tilde{\rho}$ consists of automorphisms of prescribed determinant $\phi$. This is possible since $SL_2$ is smooth. Define the cocycle $c:(G_S)^2\to Ad^0$ by $c(g_1,g_2):=\tilde{\rho}(g_1)\tilde{\rho}(g_2)\tilde{\rho}(g_1g_2)^{-1}-1$ and  the cochain $a_{\nu}:D_{\nu}\to Ad^0$ by $\tilde{\rho}(g)=(1+a_{\nu}(g))\tilde{\rho}'_{\nu}(g)$.

A short computation shows that $c|_{D_{\nu}}=\delta(a_{\nu})$. Now, we can define:
$$
([x],u):=(x,u):=\sum_{\nu\in S}inv((x_{\nu}\cup a_{\nu})+z_{\nu})
$$
where $x$ is a $1$-cocycle representing $[x]$, $z$ is a $2$-cochain of $G_S$ with values in $\mathcal{O}_S^*$ such that $\delta(z)=(x\cup c)$ (for the existence of this $2$-cochain, see the  last paragraph of the proof of \cite{neukirch} Thm. 8.6.7.), $\cup$ is the cup product followed by the map on cochains defined by the pairing $(Ad^0)^*(1)\times Ad^0\to \mathbb{F}(1)$, and for a local field $k$, the map $inv: H^2(G_k,\overline{k}^*)\to \mathbb{Q}/\mathbb{Z}$ is local Tate duality's isomorphism (see \cite{neukirch} Thm. 7.2.9).

We refer to the reader to \cite{kwannals} p. 541 for a proof of the fact that a) $(\phantom{c},\phantom{c})$ does not depend on the choice of the representative of $[x]$, nor on the choice of the lift $\tilde{\rho}$, nor on the choice of $g_{\nu}$ and b) that $(\phantom{c},\phantom{c})$ is $\mathbb{F}$-bilinear.

Now, if $f(u)=0$ with $u\neq 0$, let us extract the following subsequence from the Poitou-Tate exact sequence for $Ad^0$:
$$
H^1(S,Ad^0)\to\oplus_{\nu\in S}H^1(D_{\nu},Ad^0)/L_{\nu}\to H^1_{L_{\nu}^*}(S,(Ad^0)^*(1))^*\to\Sha^2(S,Ad^0).
$$
Since $L_{\nu}=0$ (as $l>6$), we have that $\Sha^1(S,(Ad^0)^*(1))\hookrightarrow H^1_{L_{\nu}^*}(S,(Ad^0)^*(1))$, and for each $[x]\in \Sha^1(S,(Ad^0)^*(1))$, we have $([x],u)=0$. But from Section 8.6.8 of \cite{neukirch}, $([x],u)=0$ coincides with the Poitou-Tate product of $[x]$ and the image of $[c]$ in $\Sha^2(S,Ad^0)$. Since the Poitou-Tate pairing is non-degenerate, it follows that $[c]=0$, hence $\tilde{\rho}$ can be chosen to be a Galois representation and $[z]=0$. 

Hence, the attachment $(\rho,(\rho_{\nu}),(g_{\nu}))\mapsto(\tilde{\rho},\tilde{(\rho}_{\nu}),(\tilde{g}_{\nu}))$ defines a section of the map $R_u\to R^{\square,  \mu_r}_{\mathcal{O},S}$. But this is impossible since $R_u\to R_{\mathcal{O},S}^{\square,  \mu_r}$ induces an isomorphism on $\widehat{R}_S^{\square,loc,  \mu_r}$-tangent spaces, hence the morphism is \'etale. But the section also defines a morphism on tangent spaces, hence the section is also \'etale, an in particular unramified. Hence, by, for instance, Lemma 2.10 in  \cite{milne}, it must be an isomorphism. But this is is a contradiction with \ref{eqaux2} and the fact that $I_u$ is isomorphic to $\mathbb{F}$ as $R_u$-module.

\end{proof}

\subsection{The final step.}To conclude our proof, we check in this section that $\mathcal{R}_{k_r, global}$ is finitely generated as $\mathcal{O}$-module and that there is a point $\psi_{k_r}:\mathcal{R}_{k_r, global}\to\overline{\mathbb{Z}}_l$. We also invoke a modularity result from \cite{gee} for points in the analogue of $\mathcal{R}_{k_r, global}$ over certain solvable extension, which by a base-change modularity result yields modularity over $\mathbb{Q}$, namely, that the point $\psi_{k_r}$ corresponds to a weight $k_r$ modular crystalline and potentially diagonalizable deformation of $\overline{\rho}_{f,l}$.

To start with, we recall one of the main results in \cite{gee}. This is a very general result and one of the main inputs in our proof. The relevant background on automorphic representations is standard and we refer the reader to \cite{gee} Section 1.

\begin{thm}[\cite{gee}, Theorem 4.4.1] Let $F$ be an imaginary CM field with maximal totally real subfield $F^{+}$. Let $n\in\mathbb{Z}_{\geq 1}$ be an integer, and let $l>2(n + 1)$ be an odd prime, such that  $\zeta_l\not\in F$ and all primes of $F^{+}$ above $l$ split in $F$.  Let $S$ be a finite set of finite places of $F^{+}$, including all places above $l$, such that each place in $S$ splits in $F$. For each place $\nu\in S$ choose a place $\tilde{\nu}$ of $F$ lying over $\nu$. Let $\mu$ be an algebraic character of $G_{F^{+}}$ and let $\overline{r}:G_F \to GL_n(\bar{\mathbb{F}}_l)$ be a continuous representation such that

\begin{itemize}
\item[1.] $ (\overline{r}, \overline{\mu})$ is a polarized mod $l$ representation unramified outside $S$, which either we suppose is ordinarily automorphic or we suppose is potentially diagonalizably automorphic and,
\item[2.] $\bar{r}|G_{F(\zeta_l)}$ is irreducible.
\end{itemize}

For $\nu\in S$, let $\rho_{\nu} : G_{F_{\tilde{\nu}}} \to GL_n(\mathcal{O}_{\mathbb{Q}_l})$ be a lift of $\overline{r}|_{G_{F_{\tilde{\nu}}}}$. If $\nu|l$, assume further that $\rho_{\nu}$ is potentially diagonalizable, and that for all $\tau: F_{\tilde{\nu}}\to \overline{\mathbb{Q}}_l$, $HT_{\tau}(\rho_{\nu})$ consists of $n$ distinct integers.

Then there is a regular algebraic, cuspidal, polarized automorphic representation $(\pi, \chi)$ of $GL_n(\mathbb{A}_F )$ such that
\begin{itemize}
\item[(1)] $\overline{r_{l,\iota}(\pi)}\cong \bar{r}$;
\item[(2)] $r_{l,\iota}(\chi)\varepsilon_l^{1-n} = \mu$;
\item[(3)] $\pi$ has level potentially prime to $l$;
\item[(4)] $\pi$is unramified outside $S$;
\item[(5)] for $\nu\in S$ we have: $\rho_{\nu} $  connects to $ r_{l,\iota}(\pi)|_{G_{F_{\nu}}}.$
\item[(6)]  Suppose  that for all $\nu \mid l$ the lifts $\rho_{\nu} $ are crystalline. Then for all such $\nu$ the representation $ r_{l,\iota}(\pi)|_{G_{F_{\nu}}}$ is crystalline, i.e., the level of $\pi$ is prime to $l$.
\item[(7)] Define the ring $\mathcal{R}_{F}$ as the universal deformation ring of $\bar{r}$ for the deformation problem with local conditions at primes $\nu$ dividing $\ell$ corresponding to fixing the irreducible component of the corresponding local deformation ring (of potentially crystalline representations with fixed Hodge-Tate weights) that contains $\rho_{\nu} $, and similarly for primes $\nu \in S$ not dividing $\ell$. Then $\mathcal{R}_{F}$ is a finitely generated $\mathcal{O}$-module and it has at least one point in $\overline{\mathbb{Z}}_l$, corresponding to the representation $r_{l,\iota}(\pi)$.

\end{itemize}
\label{geetool}
\end{thm}

Remark 1: The statement of this theorem in \cite{gee} does not include items (6) and (7) but these facts are proved as part of the proof of the theorem: in fact, the proof of the existence of the global lift with prescribed local properties (which is Theorem 4.3.1 of loc. cit.) begins by considering the ring $\mathcal{R}_{F}$ that we have just defined (where locally at primes $\nu$ dividing $l$ the universal local ring considered is that of crystalline representations of fixed Hodge-Tate weights if the given local representation $\rho_\nu$ is crystalline: from this fact item (6) follows), which in particular implies fixing an irreducible component at each of the local deformation rings considered. Then, it is proved that this ring has the properties that we state in item (7), and combining with a previous Automorphy Lifting Theorem it is proved that any of the characteristic 0 points in this ring is automorphic, thus proving item (1) to (5).
\\

Remark 2: As the referee has observed, all that we will need from this Theorem in what follows is the fact that the ring $\mathcal{R}_{F}$ is a finitely generated $\mathcal{O}$-module. We have decided however to include the rest of the statement because we think it is important that the reader knows that a potential version of the result that we want to prove (after base changing to a suitable CM field) appears already in \cite{gee}.
\\

We want to apply this result to the residual representation $\bar{\rho}_{f,\ell}$ as in the previous subsections. We fix $S$, as in section 3.2, to be the union of the set of primes in $N$, the prime $\ell$ and $\infty$. Choose an imaginary quadratic field $F$ in which all finite primes in $S$ are split. We consider the restriction of $\bar{\rho}_{f,\ell}$ to $G_F$ and we fix the local lifts $\rho_\nu$ of it at all places in $S$ as in previous subsections: at $\ell$ we take the potentially diagonalizable lift of Hodge-Tate weights $\{ 0 , k_r -1 \}$ constructed in lemma \ref{firstlift}, at finite places $p$ in $S$ different from $\ell$ we take a lift of $\overline{\rho_{f,l}}|_{G_{F_p}}$ satisfying the property which characterizes the local ring $R_{\mathcal{O},p}^{\square,\mu_r}$. It is important to notice that the modular representation that will be produced as output of the theorem will have the same behaviour (more precisely, the same local inertial type) locally at all primes not dividing $l$ in $S$, as follows from item (5) or (7) because inertial types are constant in irreducible components of these local deformation rings (see Lemma 1.3.4 in \cite{gee}).\\

Condition 1 in the theorem is satisfied since we know that the modular lift $\rho_{f,\ell}$ is potentially diagonalizable, and by base change the restriction to $G_F$ of this representation is  automorphic  (\cite{langlands}, assertion (A) of page 19) and also potentially diagonalizable, as this condition is obviously preserved for restriction base change. 

Condition 2 is also satisfied because we are assuming that the image of $\bar{\rho}_{f,\ell}$ contains $SL_2(\mathbb{F}_l)$, a condition that is preserved when restricting to $G_{F(\zeta_l)}$.  

 With the local conditions that we have imposed, we apply the previous theorem to obtain a lift of the restriction of $\bar{\rho}_{f,\ell}$ to $G_F$ which on the one hand is automorphic and on the other hand corresponds to a point in the deformation ring $\mathcal{R}_{k_r,global, F}$ which is defined as in previous subsections except for the fact that we are now working  over $F$. This lift is attached to an automorphic form $\pi$ of $GL_2(F)$ and it has all the properties (locally at $l$ it is crystalline, potentially diagonalizable and has the right Hodge-Tate weights) that we want in the main result of this paper except for the fact that it is a representation of $G_F$, not necessarily extending to $G_\Q$. \\
We also know thanks to item (7) of the theorem that the ring $\mathcal{R}_{k_r,global, F}$ is a finitely generated
$\mathcal{O}$-module. From this the next step is to ``descend" this property to the corresponding deformation ring over $\Q$ of $\bar{\rho}_{f,\ell}$.

\begin{prop} If  $\mathcal{R}_{k_r,global,F}$ is a finitely generated $\mathcal{O}$-module, then $\mathcal{R}_{k_r,global}$ is a finitely generated $\mathcal{O}$-module.
\label{fgeneration}
\end{prop}
\begin{proof}
We follow a similar argument to that used in the proof of Theorem 10.1 of \cite{kwannals}, which is stated for the universal deformation ring corresponding to another deformation problem, but the proof can be adapted to our setting as follows:

Since we are assuming that $\mathcal{R}_{k_r,global, F}$ is a finitely generated $\mathcal{O}$-module, then $\mathcal{R}_{k_r,global,F}/\frak{L}$ is a finite set, where $\frak{L} = (\ell)$. Hence, $\mathrm{GL}_2(\mathcal{R}_{k_r,global,F}/\frak{L})$ is also a finite set. 

Now, from the universal property of $\mathcal{R}_{k_r,global,F}/\frak{L}$, there exists a CNL $\mathcal{O}$-algebra morphism $\gamma:\mathcal{R}_{k_r,global,F}/\frak{L}\to\mathcal{R}_{k_r,global}/\frak{L}$ which takes the universal mod $\frak{L}$ representation $\alpha:G_F\to GL_2(\mathcal{R}_{k_r,global,F}/\frak{L})$ to the restriction to $G_F$ of the universal mod $\frak{L}$ representation $\beta:G_{\mathbb{Q}}\to GL_2(\mathcal{R}_{k_r,global}/\frak{L})$.

Since  $\alpha$ has finite image, so has $\beta|_{G_F}$. Since $G_F$ has finite index in $G_\Q$,  the image of  $\beta: G_{\mathbb{Q}}\to GL_2(\mathcal{R}_{k_r,global}/\frak{L})$ is also finite. From this, since $\overline{\rho}_{f,l}$ is absolutely irreducible, using Lemma 3.6 of \cite{kwp2}, we conclude that $\mathcal{R}_{k_r,global}$ is a finitely generated $\mathcal{O}$-module.
\end{proof}

\begin{cor}There exists at least one point $\mathcal{R}_{k_r,global}\to\bar{\Z}_\ell\to 0$ where $\bar{\Z}_\ell$ denotes the ring of integers of $\bar{\Q}_\ell$.
\end{cor}

\begin{proof} We know from item (7) of theorem \ref{geetool} that $\mathcal{R}_{k_r,global,F}$ is a finitely generated $\mathcal{O}$-module, thus it follows from proposition \ref{fgeneration} that 
$\mathcal{R}_{k_r,global}$ is a finitely generated $\mathcal{O}$-module.
As proved by B\"ockle (see \cite{bockle}, Lemma 2) from the combination of this with theorem \ref{cotadimension} it follows that  $\mathcal{R}_{k_r,global}$ is a finite flat and complete intersections $\mathcal{O}$-module. From this, we deduce that it contains a point over $\bar{\Z}_\ell$.
\end{proof}

Now we conclude our work with the following observation:

\begin{cor}The point $\mathcal{R}_{k_r,global}\to \bar{\Z}_\ell \to 0$ corresponds to a modular representation of weight $k_r$.
\end{cor}
\begin{proof} Since the restriction to $G_F$ of this lift is modular due to Theorem 4.2.1 in \cite{gee}, the lift itself is modular by solvable base change (\cite{clozel}, Theorem 6.2 for $n=2)$.
\end{proof}

\end{document}